\documentclass[11pt]{amsart}

\author{Carlo Sanna}
\thanks{\mbox{The author is supported by a postdoctoral fellowship of INdAM and is a member of the INdAM group GNSAGA}}
\address{Universit\`a degli Studi di Genova\\Department of Mathematics\\Genova, Italy}
\email{carlo.sanna.dev@gmail.com}

\keywords{prime factorization; squarefree numbers; powerful number}
\subjclass[2010]{Primary: 11N25, Secondary: 11N37, 11N64.}
\title[On the number of distinct exponents in the prime factorization of an integer]{On the number of distinct exponents in\\ the prime factorization of an integer}

\usepackage{amsmath}
\usepackage{amssymb}
\usepackage{amsthm}
\usepackage{geometry}
\usepackage{color}
\usepackage{hyperref}
\usepackage{enumerate}
\hypersetup{colorlinks=true}

\newtheorem{thm}{Theorem}[section]
\newtheorem{cor}{Corollary}[section]
\newtheorem{lem}[thm]{Lemma}
\theoremstyle{remark}

\begin{document}

\begin{abstract}
Let $f(n)$ be the number of distinct exponents in the prime factorization of the natural number $n$.
We prove some results about the distribution of $f(n)$.
In particular, for any positive integer $k$, we obtain that
\begin{equation*}
\#\{n \leq x : f(n) = k\} \sim A_k x
\end{equation*}
and
\begin{equation*}
\#\{n \leq x : f(n) = \omega(n) - k\} \sim \frac{B x (\log \log x)^k}{k! \log x} ,
\end{equation*}
as $x \to +\infty$, where $\omega(n)$ is the number of prime factors of $n$ and $A_k, B > 0$ are some explicit constants.
The latter asymptotic extends a result of Akta\c{s} and Ram~Murty about numbers having mutually distinct exponents in their prime factorization.
\end{abstract}

\maketitle

\section{Introduction}

Let $n = p_1^{a_1} \cdots p_s^{a_s}$ be the factorization of the natural number $n > 1$, where $p_1 < \cdots < p_s$ are prime numbers and $a_1, \dots, a_s$ are positive integers.
Several functions of the exponents $a_1, \ldots, a_s$ have been studied, including: their product~\cite{MR0291104}, their arithmetic mean~\cite{MR1319087, MR0332683, MR567210, MR0252311}, and their maximum and minimum~\cite{MR2212332, MR0241373, MR2265998, MR0437438}.
See also~\cite{MR1436027, MR0279056} for more general functions.

Let $f$ be the arithmetic function defined by $f(1) := 0$ and $f(n) := \#\{a_1, \dots, a_s\}$ for all natural numbers $n > 1$.
In other words, $f(n)$ is the number of distinct exponents in the prime factorization of $n$.
The first values of $f(n)$ are listed in sequence A071625 of OEIS~\cite{OEIS}.

Our first contribution is a quite precise result about the distribution of $f(n)$.

\begin{thm}\label{thm:fdistribution}
There exists a sequence of positive real numbers $(A_k)_{k \geq 1}$ such that, given any arithmetic function $\phi$ satisfying $|\phi(k)| < a^k$ for some fixed $a > 1$, we have that the series
\begin{equation}\label{equ:Mphi}
M_\phi := \sum_{k \,=\, 1}^\infty A_k \phi(k)
\end{equation}
converges and
\begin{equation*}
\sum_{n \,\leq\, x} \phi(f(n)) = M_\phi x + O_{a,\varepsilon}(x^{1/2 + \varepsilon}) ,
\end{equation*}
for all $x \geq 1$ and $\varepsilon > 0$.
\end{thm}

From Theorem~\ref{thm:fdistribution} it follows immediately that all the moments of $f$ are finite and that $f$ has a limiting distribution.
In particular, we highlight the following corollary:

\begin{cor}\label{cor:fAk}
For each positive integer $k$, we have
\begin{equation*}
\#\{n \leq x : f(n) = k\} = A_k x + O_{\varepsilon}(x^{1/2 + \varepsilon}) ,
\end{equation*}
for all $x \geq 1$ and $\varepsilon > 0$.
\end{cor}

We provides also a formula for $A_k$.
Before stating it, we need to introduce some notation.
Let $\psi$ be the Dedekind function, defined by
\begin{equation*}
\psi(n) := n \prod_{p \,\mid\, n} \left(1 + \frac1{p}\right)
\end{equation*}
for each positive integer $n$, and let $(\rho_k)_{k \geq 1}$ be the family of arithmetic functions supported on squarefree numbers and satisfying
\begin{equation*}
\rho_1(n) = \begin{cases}1 & \text{ if } n = 1, \\ 0 & \text{ if } n > 1,\end{cases} \quad\quad \rho_{k + 1}(n) = \begin{cases}0 & \text{ if } n = 1, \\ \frac1{n - 1} \sum_{\substack{d \,\mid\, n \\ d \,<\, n}} \rho_k(d) &\text{ if } n > 1,\end{cases}
\end{equation*}
for all squarefree numbers $n$ and positive integers $k$.

\begin{thm}\label{thm:Ak}
We have
\begin{equation*}
A_k = \frac{6}{\pi^2} \sum_{n \,=\, 1}^{\infty} \frac{\rho_k(n)}{\psi(n)} 
\end{equation*}
for each positive integer $k$.
\end{thm}

Clearly, $f(n) \leq \omega(n)$ for all positive integers $n$, where $\omega(n)$ denotes the number of prime factors of $n$.
Motivated by a question of Recam{\'a}n Santos~\cite{MO}, 
Akta\c{s} and Ram~Murty~\cite{MR3660343} studied the natural numbers $n$ such that all the exponents in their prime factorization are distinct, that is, $f(n) = \omega(n)$.
They called such numbers \emph{special numbers} (sequence A130091 of OEIS~\cite{OEIS}) and they proved the following:

\begin{thm}\label{thm:special}
The number of special numbers not exceeding $x$ is
\begin{equation*}
\frac{Bx}{\log x} + O\!\left(\frac{x}{(\log x)^2}\right) ,
\end{equation*}
for all $x \geq 2$, where
\begin{equation*}
B := \sum_{\ell} \frac1{\ell}
\end{equation*}
and the sum of over natural numbers $\ell$ that are powerful and special.
\end{thm}

Let $g$ be the arithmetic function defined by $g(n) := \omega(n) - f(n)$ for all positive integers $n$.
Hence, by the previous observation, $g$ is a nonnegative function and $g(n) = 0$ if and only if $n$ is a special number.
We prove the following result about $g$, which extends Theorem~\ref{thm:special} and it is somehow dual to Corollary~\ref{cor:fAk}.

\begin{thm}\label{thm:gdistribution}
For each nonnegative integer $k$, we have
\begin{equation*}
\#\{n \leq x : g(n) = k\} = \frac{Bx(\log \log x)^k}{k!\log x} \left(1 + O_k\!\left(\frac1{\log \log x}\right)\right) ,
\end{equation*}
for all $x \geq 3$.
\end{thm}

\subsection*{Notation}
We employ the Landau--Bachmann ``Big Oh'' notation $O$, as well as the associated Vinogradov symbol $\ll$, with their usual meanings. 
Any dependence of the implied constants is explicitly stated.
We reserve the letter $p$ for prime numbers.

\section{Preliminaries}

Recall that a natural number $n$ is called \emph{powerful} if $p \mid n$ implies $p^2 \mid n$, for all primes $p$.
For~all $x \geq 1$, let~$\mathcal{P}(x)$ be the set of powerful numbers not exceeding $x$.

\begin{lem}\label{lem:powerful}
We have $\#\mathcal{P}(x) \ll x^{1/2}$ for every $x \geq 1$.
\end{lem}
\begin{proof}
See~\cite{MR0266878}.
\end{proof}

\begin{lem}\label{lem:ellsums}
We have
\begin{equation*}
\sum_{\substack{\ell \,\in\, \mathcal{P} \\ \ell \,>\, y}} \frac1{\ell} \ll \frac1{y^{1/2}} , \quad \sum_{\ell \,\in\, \mathcal{P}(y)} \frac1{\ell^{1/2}} \ll \log y,
\end{equation*}
for all $y \geq 2$.
\end{lem}
\begin{proof}
By Lemma~\ref{lem:powerful} and by partial summation, we have
\begin{equation*}
\sum_{\substack{\ell \,\in\, \mathcal{P} \\ \ell \,>\, y}} \frac1{\ell} = \left.\frac{\#\mathcal{P}(t)}{t}\right|_{t \,=\, y}^{+\infty} + \int_y^{+\infty} \frac{\#\mathcal{P}(t)}{t^2}\,\mathrm{d} t \ll \int_y^{+\infty} \frac{\mathrm{d} t}{t^{1+1/2}} \ll \frac1{y^{1/2}} .
\end{equation*}
The proof of the second claim is similar.
\end{proof}

We need the following upper bound for the number of prime factors of a natural number.

\begin{lem}\label{lem:omega}
We have
\begin{equation*}
\omega(n) \ll \frac{\log n}{\log \log n}
\end{equation*}
for all integers $n \geq 3$.
\end{lem}
\begin{proof}
See, e.g.,~\cite[Proposition~7.10]{MR2919246}.
\end{proof}

For every $x \geq 1$ and every positive integer $h$, let $Q(x;h)$ denote the number of squarefree numbers not exceeding $x$ and relatively prime with $h$.

\begin{lem}\label{lem:squarefree}
We have
\begin{equation*}
Q(x;h) = \frac{6}{\pi^2}\frac{h}{\psi(h)} \,x + O\!\left(4^{\omega(h)}(x^{1/2} + 1)\right)
\end{equation*}
for all $x \geq 1$ and all positive integers $h$.
\end{lem}
\begin{proof}
It follows easily from \cite[Eq.~8]{MR0374056}.
\end{proof}

For every $x \geq 1$ and every positive integers $s,h$, let $Q_s(x;h)$ denote the number of squarefree numbers not exceeding $x$, having exactly $s$ prime factors, and relatively prime with $h$.

\begin{lem}\label{lem:landau}
We have
\begin{equation*}
Q_s(x;h) = \frac{x(\log \log x)^{s-1}}{(s-1)!\log x}\left(1 + O_{\delta,s}\!\left(\frac{\log \log (h+2)}{\log \log x}\right)\right)
\end{equation*}
for all $x \geq 3$, $0 < \delta < 1$, and for all integers $1 \leq h \leq x^{\delta}$ and $s \geq 1$.
\end{lem}
\begin{proof}
For $s = 1$ the claim follows from the Prime Number Theorem, while for $h = 1$ the claim is a classic result of Landau~\cite{MR1504359}.
Hence, suppose $s,h > 1$.
Also, we can assume $x \geq 3^{1/(1-\delta)}$.
If $n \leq x$ is a squarefree number having exactly $s$ prime factors and such that $(n, h) > 1$, then $n = pn^\prime$ where $p$ is a prime number dividing $h$ and $n^\prime \leq x/p$ is a squarefree number having exactly $s - 1$ prime factor.
Therefore,
\begin{align*}
0 &\leq Q_s(x;1) - Q_s(x;h) \leq \sum_{p \,\mid\, h} Q_{s-1}\!\left(\frac{x}{p}, 1\right) \ll_s \sum_{p \,\mid\, h}\frac{x}{p}\frac{(\log \log (x/p))^{s-2}}{\log(x/p)} \\
&\ll_{\delta} \frac{x(\log \log x)^{s-2}}{\log x}\sum_{p \,\mid\, h}\frac1{p} \ll \frac{x(\log \log x)^{s-1}}{\log x}\frac{\log \log (h+2)}{\log \log x} ,
\end{align*}
where we used the fact that $p \leq x^{\delta}$ and Mertens' second theorem~\cite[Theorem~4.5]{MR2919246}.
Consequently,
\begin{align*}
Q_s(x;h) &= Q_s(x;1) + O_{\delta,s}\!\left(\frac{x(\log \log x)^{s-1}}{\log x}\frac{\log \log (h+2)}{\log \log x}\right) \\
&= \frac{x(\log \log x)^{s-1}}{(s-1)!\log x} + O_{\delta,s}\!\left(\frac{x(\log \log x)^{s-1}}{\log x}\frac{\log \log (h+2)}{\log \log x}\right) ,
\end{align*}
as claimed.
\end{proof}

Finally, we need a lemma about certain sums of powers.

\begin{lem}\label{lem:x1xk}
Let $a_0$ be an integer.
For all $x_1, \dots, x_k > 1$ we have
\begin{equation*}
\sum_{a_0 \,<\, a_1 \,<\, \cdots \,<\, a_k} \frac1{x_1^{a_1} \cdots x_k^{a_k}} = \frac1{(x_1 \cdots x_k)^{a_0}} \prod_{j \,=\, 1}^k \frac1{x_j \cdots x_k - 1} ,
\end{equation*}
where the sum is over all integers $a_1, \dots, a_k$ satisfying $a_0 \,<\, a_1 \,<\, \cdots \,<\, a_k$.
\end{lem}
\begin{proof}
We proceed by induction on $k$.
For $k = 1$, we have
\begin{equation}\label{equ:x1a1}
\sum_{a_0 \,<\, a_1} \frac1{x_1^{a_1}} = \frac1{x_1^{a_0 + 1}} \sum_{d \,=\, 0}^\infty \frac1{x_1^d} = \frac1{x_1^{a_0}} \frac1{x_1 - 1} ,
\end{equation}
as claimed.
Supposing that the claim is true for $k$, we shall prove it for $k + 1$.
We have
\begin{align*}
\sum_{a_0 \,<\, \cdots \,<\, a_{k+1}}& \frac1{x_1^{a_1} \cdots x_{k+1}^{a_{k+1}}} = \sum_{a_0 \,<\, \cdots \,<\, a_k} \frac1{x_1^{a_1} \cdots x_k^{a_k}} \sum_{a_k \,<\, a_{k+1}} \frac1{x_{k+1}^{a_{k+1}}} \\
&= \sum_{a_0 \,<\, \cdots \,<\, a_{k+1}} \frac1{x_1^{a_1} \cdots x_{k-1}^{a_{k-1}} (x_k x_{k+1})^{a_k}}  \frac1{x_{k+1} - 1} \\
&= \frac1{(x_1 \cdots x_{k+1})^{a_0}} \prod_{j \,=\, 1}^k \frac1{x_j \cdots x_{k+1} - 1} \frac1{x_{k+1}-1} \\
&= \frac1{(x_1 \cdots x_{k+1})^{a_0}} \prod_{j \,=\, 1}^{k+1} \frac1{x_j \cdots x_{k+1} - 1} ,
\end{align*}
where we used~\eqref{equ:x1a1}, with $a_0$ and $x_1$ replaced respectively by $a_k$ and $x_{k+1}$, and the induction hypothesis.
\end{proof}

\section{Proof of Theorem~\ref{thm:fdistribution}}

We begin by proving that for each positive integer $k$ there exists $A_k > 0$ such that
\begin{equation}\label{equ:Nkxdef}
N_k(x) := \#\{n \leq x : f(n) = k\} = A_k x + O_\varepsilon(x^{1/2 + \varepsilon/2}) ,
\end{equation}
for all $x \geq 1$ and $\varepsilon > 0$.
Clearly, every natural number $n$ can be written in a unique way as $n = m\ell$, where $m$ is a squarefree number, $\ell$ is a powerful number, and $(m,\ell) = 1$.
If $m = 1$ then $n = \ell$ is powerful and, by Lemma~\ref{lem:powerful}, belongs to a set of cardinality $O(x^{1/2})$.
If $m > 1$ then $f(n) = k$ is equivalent to $f(\ell) = k - 1$.
Also, for each $\ell$ there are exactly $Q(x/\ell;\ell) - 1$ choices for $m > 1$.
Therefore, we have
\begin{equation}\label{equ:Nkx}
N_k(x) = \sum_{\substack{\ell \,\in\, \mathcal{P}(x) \\ f(\ell) \,=\, k - 1}} \left(Q\!\left(\frac{x}{\ell}; \ell\right) - 1\right) + O(x^{1/2}) ,
\end{equation}
for all $x \geq 1$.
For each positive integer $\ell \leq x$, Lemma~\ref{lem:omega} gives $4^{\omega(\ell)} \ll_\varepsilon x^{\varepsilon/4}$.
Consequently, by Lemma~\ref{lem:squarefree}, we obtain
\begin{equation}\label{equ:Qxll}
Q\!\left(\frac{x}{\ell};\ell\right) = \frac{6}{\pi^2}\frac{x}{\psi(\ell)} + O_\varepsilon\!\left(\frac{x^{1/2+\varepsilon/4}}{\ell^{1/2}}\right) ,
\end{equation}
for all positive integers $\ell \leq x$.
By Lemma~\ref{lem:ellsums}, we have
\begin{equation}\label{equ:part1}
\sum_{\substack{\ell \,\in\, \mathcal{P} \\ \ell \,>\, x}} \frac1{\psi(\ell)} < \sum_{\substack{\ell \,\in\, \mathcal{P} \\ \ell \,>\, x}} \frac1{\ell} \ll \frac1{x^{1/2}} ,
\end{equation}
for all $x \geq 1$.
In particular, the series
\begin{equation}\label{equ:Akfell}
A_k := \frac{6}{\pi^2}\sum_{\substack{\ell \,\in\, \mathcal{P} \\ f(\ell) \,=\, k - 1}} \frac1{\psi(\ell)}
\end{equation}
converges.
Also, again by Lemma~\ref{lem:ellsums}, we have
\begin{equation}\label{equ:part2}
\sum_{\ell \,\in\, \mathcal{P}(x)} \frac1{\ell^{1/2}} \ll \log x \ll_\varepsilon x^{\varepsilon / 4} .
\end{equation}
At this point, putting together~\eqref{equ:Nkx} and~\eqref{equ:Qxll}, and using~\eqref{equ:part1} and~\eqref{equ:part2}, we obtain
\begin{align*}
N_k(x) &= \sum_{\substack{\ell \,\in\, \mathcal{P}(x) \\ f(\ell) \,=\, k - 1}} \left(\frac{6}{\pi^2}\frac{x}{\psi(\ell)} + O_\varepsilon\!\left(\frac{x^{1/2+\varepsilon/4}}{\ell^{1/2}}\right)\right) + O(x^{1/2}) \\
&= A_k x + O\!\left(\sum_{\substack{\ell \,\in\, \mathcal{P} \\ \ell \,>\, x}} \frac{x}{\psi(\ell)}\right) + O_\varepsilon\!\left(\sum_{\ell \,\in\, \mathcal{P}(x)} \frac{x^{1/2+\varepsilon/4}}{\ell^{1/2}}\right) + O(x^{1/2}) \\
&= A_k x + O_\varepsilon(x^{1/2 + \varepsilon/2}) ,
\end{align*}
as desired.
Thus~\eqref{equ:Nkxdef} is proved.

Now we shall show that
\begin{equation}\label{equ:Akbound}
A_k \leq \frac{6}{\pi^2}\frac1{(k-1)!}
\end{equation}
for all positive integers $k$.
For $k = 1$ the claim is obvious since $A_1 = 6 / \pi^2$.
Hence, assume $k \geq 2$.
If $\ell$ is a powerful number such that $f(\ell) = k-1$, then $\ell = m_1^{a_1} \cdots m_{k-1}^{a_{k - 1}}$ for some integers $m_1, \dots, m_{k-1} \geq 2$ and $2 \leq a_1 < \cdots < a_{k-1}$.
Consequently,
\begin{align*}
\frac{\pi^2}{6} A_k &= \sum_{\substack{\ell \,\in\, \mathcal{P} \\ f(\ell) \,=\, k-1}} \frac1{\psi(\ell)} < \sum_{\substack{\ell \,\in\, \mathcal{P} \\ f(\ell) \,=\, k-1}} \frac1{\ell} < \prod_{j \,=\, 1}^{k - 1} \sum_{m \,=\, 2}^\infty \sum_{a \,=\, j + 1}^\infty \frac1{m^a} \\
&= \prod_{j \,=\, 1}^{k - 1} \sum_{m \,=\, 2}^\infty \frac1{m^j(m - 1)} \leq \prod_{j \,=\, 1}^{k - 1} \frac1{j} = \frac1{(k-1)!} ,
\end{align*}
where we used the facts that
\begin{equation*}
\sum_{m \,=\, 2}^\infty \frac1{m(m - 1)} = \sum_{m \,=\, 2}^\infty \left(\frac1{m-1}-\frac1{m}\right) = 1
\end{equation*}
and
\begin{align*}
\sum_{m \,=\, 2}^\infty \frac1{m^j(m - 1)} &< \frac1{2^j} + \frac1{3^j\cdot 2} + \sum_{n \,=\, 3}^\infty \frac1{n^{j+1}} \\
&< \frac1{2^j} + \frac1{3^j\cdot 2} + \int_2^{+\infty} \frac{\mathrm{d}t}{t^{j+1}} = \frac1{2^j} + \frac1{3^j\cdot 2} + \frac1{j2^j} < \frac1{j} ,
\end{align*}
for all integers $j \geq 2$.
Thus~\eqref{equ:Akbound} is proved.

Now let $\phi$ be an arithmetic function satisfying $|\phi(k)| < a^k$ for all positive integers $k$, where $a > 1$ is some constant.
From~\eqref{equ:Akbound} it follows that series~\eqref{equ:Mphi} converges.
Define 
\begin{equation*}
y := 2a + \lfloor C \log x / \log \log (x+2) \rfloor,
\end{equation*}
where $C > 0$ is some absolute constant.
Since $f(n) \leq \omega(n)$ for all positive integers $n$, by Lemma~\ref{lem:omega}, we can choose $C$ sufficiently large so that $f(n) \leq y$ for all natural numbers $n \leq x$.
Moreover, from~\eqref{equ:Akbound} and $y \geq 2a$, we get that
\begin{equation}\label{equ:Akphitail}
\sum_{k \,>\, y} A_k \phi(k) \ll \sum_{k \,>\, y} \frac{a^k}{(k-1)!} < \frac{a^{y+1}}{y!} \sum_{j \,=\, 0}^{\infty} \left(\frac{a}{y}\right)^j \ll_a \frac{a^y}{y!} \ll_a \frac1{x^{1/2}} ,
\end{equation}
and
\begin{equation}\label{equ:ayy}
a^y y \ll_{a,\varepsilon} x^{\varepsilon / 2} ,
\end{equation}
for all $x \geq 1$.
Therefore, putting together~\eqref{equ:Nkxdef}, \eqref{equ:Akphitail}, and~\eqref{equ:ayy}, we have
\begin{align*}
\sum_{n \,\leq\, x} \phi(f(n)) &= \sum_{k \,\leq\, y} N_k(x) \phi(k) = \sum_{k \,\leq\, y} \left(A_k \phi(k) x + O_\varepsilon(\phi(k)x^{1/2 + \varepsilon/2})\right) \\
&= M_\phi x + O\!\left(\sum_{k \,>\, y} A_k \phi(k) x \right) + O_\varepsilon(a^y y x^{1/2 + \varepsilon/2}) = M_\phi x + O_{a,\varepsilon}(x^{1/2 + \varepsilon}) ,
\end{align*}
for all $x \geq 1$ and $\varepsilon > 0$.
The proof is complete.

\section{Proof of Theorem~\ref{thm:Ak}}

Recall that $A_k$ is defined by~\eqref{equ:Akfell}.
For $k = 1$ the claim is obvious, since $f(\ell) = 0$ if and only if $\ell = 1$.
Hence, assume $k \geq 2$.
If $\ell$ is a powerful number such that $f(\ell) = k - 1$, then $\ell$ can be written in a unique way as $\ell = m_1^{a_1} \cdots m_{k-1}^{a_{k-1}}$, where $1 < a_1 < \cdots < a_{k-1}$ are integers and $m_1, \dots, m_{k-1} > 1$ are pairwise coprime squarefree numbers.
Therefore, from~\eqref{equ:Akfell} and Lemma~\ref{lem:x1xk} we obtain
\begin{align*}
\frac{\pi^2}{6}A_k &= \sum_{m_1, \dots, m_{k-1}} \sum_{1 \,<\, a_1 \,<\, \cdots \,<\, a_{k-1}} \frac1{\psi(m_1^{a_1} \cdots m_{k-1}^{a_{k-1}})} \\
&= \sum_{m_1, \dots, m_{k-1}} \frac{m_1 \cdots m_{k-1}}{\psi(m_1 \cdots m_{k-1})} \sum_{1 \,<\, a_1 \,<\, \cdots \,<\, a_{k-1}} \frac1{m_1^{a_1} \cdots m_{k-1}^{a_{k-1}}} \\
&= \sum_{m_1, \dots, m_{k-1}} \frac1{\psi(m_1 \cdots m_{k-1})} \prod_{j \,=\, 1}^{k-1} \frac1{m_j \cdots m_{k-1} - 1} ,
\end{align*}
where, here and for the rest of the proof, in summation subscripts $m_1, \dots, m_{k-1}$ are meant to be pairwise coprime, squarefree, and greater than $1$.
At this point, it is enough to prove that
\begin{equation*}
\sum_{n \,=\, m_1 \cdots m_{k-1}} \prod_{j \,=\, 1}^{k-1} \frac1{m_j \cdots m_{k-1} - 1} = \rho_k(n)
\end{equation*}
for all squarefree numbers $n > 1$.
We proceed by induction on $k$.
For $k = 2$, the claim is true since
\begin{equation*}
\frac1{n - 1} = \frac{\rho_1(1)}{n - 1} = \frac1{n - 1}\sum_{\substack{d \,\mid\, n \\ d \,<\, n}} \rho_1(d) = \rho_2(n) ,
\end{equation*}
for all squarefree numbers $n > 1$.
Assuming that the claim is true for $k$, we shall prove it for $k + 1$.
We have
\begin{align*}
\sum_{n \,=\, m_1 \cdots m_k}& \prod_{j \,=\, 1}^{k} \frac1{m_j \cdots m_{k} - 1} = \frac1{n - 1} \sum_{m_1 \,\mid\, n} \sum_{n/m_1 \,=\, m_2 \cdots m_k} \prod_{j \,=\, 2}^k \frac1{m_j \cdots m_{k} - 1} \\
&= \frac1{n - 1} \sum_{m_1 \,\mid\, n} \rho_k(n / m_1) = \frac1{n - 1} \sum_{\substack{d \,\mid\, n \\ d \,<\, n}} \rho_k(d) = \rho_{k+1}(n) ,
\end{align*}
for all squarefree numbers $n > 1$, as desired.
The proof is complete.

\section{Proof of Theorem~\ref{thm:gdistribution}}

We have to count the number of positive integers $n \leq x$ such that $g(n) = k$.
As in the proof of Theorem~\ref{thm:fdistribution}, every $n$ can be written in a unique way as $n = m\ell$, where $m$ is a squarefree number, $\ell$ is a powerful number, and $(m,\ell) = 1$.
If $m = 1$ then $n = \ell$ is powerful and, by Lemma~\ref{lem:powerful}, belongs to a set of cardinality $O(x^{1/2})$.
If $m > 1$ then 
\begin{equation*}
\omega(m) = \omega(n) - \omega(\ell) = g(n) + f(n) - f(\ell) - g(\ell) = k + 1 - g(\ell).
\end{equation*}
In particular, $1 \leq \omega(m) \leq k + 1$.
Assume $x$ sufficiently large, and put $y := (\log x)^2$.
Then, by Lemma~\ref{lem:ellsums}, the number of $n \leq x$ such that $\ell > y$ is at most
\begin{equation*}
\sum_{\substack{\ell \,\in\, \mathcal{P} \\ \ell \,>\, y}} \frac{x}{\ell} \ll \frac{x}{y^{1/2}} = \frac{x}{\log x} .
\end{equation*}
Therefore,
\begin{equation}\label{equ:MkQs}
M_k(x) := \#\{n \leq x : g(n) = k\} = \sum_{s \,=\, 1}^{k + 1} \sum_{\substack{\ell \,\in\,\mathcal{P}(y) \\ g(\ell) \,=\, k + 1 - s}} Q_s\!\left(\frac{x}{\ell};\ell\right) + O\!\left(\frac{x}{\log x}\right) .
\end{equation}
For each nonnegative integer $r$, put
\begin{equation*}
B_r := \sum_{\substack{\ell \,\in\, \mathcal{P} \\ g(\ell) \,=\, r}} \frac1{\ell} .
\end{equation*}
Note that, in light of Lemma~\ref{lem:ellsums}, the series defining $B_r$ converges and, more precisely,
\begin{equation}\label{equ:Brtail}
\sum_{\substack{\ell \,\in\, \mathcal{P}(y) \\ g(\ell) \,=\, r}} \frac1{\ell} = B_r + O\!\left(\frac1{y^{1/2}}\right) = B_r + O\!\left(\frac1{\log x}\right).
\end{equation}
Clearly, we can assume $x$ sufficiently large so that $x / y \geq 3$ and $y \leq x^{\delta/(1+\delta)}$, for some fixed $0 < \delta < 1$.
Hence, applying Lemma~\ref{lem:landau} we obtain
\begin{align*}
Q_s\!\left(\frac{x}{\ell};\ell\right) &= \frac{x(\log \log (x/\ell))^{s-1}}{\ell(s-1)!\log (x / \ell)}\left(1 + O_k\!\left(\frac{\log \log (\ell+2)}{\log \log (x/\ell)}\right)\right) \\
&= \frac{x(\log \log x)^{s-1}}{\ell(s-1)!\log x}\left(1 + O_k\!\left(\frac{\log \ell}{\log x}\right)\right)\left(1 + O_k\!\left(\frac{\log \log (\ell+2)}{\log \log x}\right)\right) \\
&= \frac{x(\log \log x)^{s-1}}{\ell(s-1)!\log x}\left(1 + O_k\!\left(\frac{\log (\ell+1)}{\log \log x}\right)\right) ,
\end{align*}
for all positive integers $s \leq k + 1$ and $\ell \leq y$.
Consequently,
\begin{align}\label{equ:sumQs}
\sum_{\substack{\ell \,\in\, \mathcal{P}(y) \\ g(\ell) \,=\, k + 1 - s}} Q_s\!\left(\frac{x}{\ell};\ell\right) &= \frac{x(\log \log x)^{s-1}}{(s-1)!\log x}\sum_{\substack{\ell \,\in\, \mathcal{P}(y) \\ g(\ell) \,=\, k + 1 - s}}\frac1{\ell}\left(1 + O_k\!\left(\frac{\log (\ell + 1)}{\log \log x}\right)\right) \\
&= \frac{x(\log \log x)^{s-1}}{(s-1)!\log x}\left(B_{k+1-s} + O\!\left(\frac1{\log x}\right) + O_k\!\left(\frac1{\log \log x}\right)\right) \nonumber\\
&= \frac{x(\log \log x)^{s-1}}{(s-1)!\log x}\left(B_{k+1-s} + O_k\!\left(\frac1{\log \log x}\right)\right) , \nonumber
\end{align}
where we used~\eqref{equ:Brtail} and the fact that the series
\begin{equation*}
\sum_{\ell \,\in\, \mathcal{P}}\frac{\log (\ell+1)}{\ell} 
\end{equation*}
converges.
Thus, putting together~\eqref{equ:MkQs} and~\eqref{equ:sumQs}, and noting that $B_0 = B$, we obtain
\begin{equation*}
M_k(x) = \frac{B x(\log \log x)^k}{k!\log x}\left(1 + O_k\!\left(\frac1{\log \log x}\right)\right) ,
\end{equation*}
as desired.
The proof is complete.

\bibliographystyle{amsplain}
\providecommand{\bysame}{\leavevmode\hbox to3em{\hrulefill}\thinspace}

\end{document}